\documentclass[12pt]{amsart}

\usepackage{graphicx}
\usepackage{amsmath,amssymb,amsthm,enumerate,stmaryrd}
\usepackage[colorlinks=true,citecolor=black,linkcolor=black,urlcolor=blue]{hyperref}
\usepackage[mathlines]{lineno}
\usepackage[section]{placeins}
\usepackage{tikz}

\usetikzlibrary{arrows}

\tikzset{
  every picture/.style=thick,
  every node/.style={circle, draw, inner sep=3pt, outer sep=0},
  every label/.style={draw=none, inner sep=0pt, outer sep=0},
  bluenode/.style={circle, draw=black, fill=blue!40, very thick, minimum size=6mm},
  whitenode/.style={circle, draw=black, fill=black!10, very thick, minimum size=6mm},
  squarednode/.style={rectangle, draw=red!60, fill=red!5, very thick, minimum size=5mm},
  every loop/.style={min distance=8mm},
  label distance=0.1cm
} 

\newtheorem{theorem}{Theorem}[section]
\newtheorem{lemma}[theorem]{Lemma}

\newtheorem{proposition}[theorem]{Proposition}

\theoremstyle{definition}
\newtheorem{definition}[theorem]{Definition}

\theoremstyle{remark}
\newtheorem{remark}[theorem]{Remark}

\numberwithin{equation}{section}

\title{}

\author[Lozano, Nasserasr, Wall]{ Jorge~Lozano, Shahla Nasserasr, Thomas~Wall}

\address[Jorge~Lozano]{~}
\email{jlozano1@ibm.com}
\address[Shahla~Nasserasr, Thomas~Wall]{School  of Mathematics \& Statistics, Rochester Institute of Technology, Rochester, NY, USA}
\email{shahla@mail.rit.edu, tjw3002@rit.edu}
\title{Claw-free cubic graphs and zero forcing}

\begin{document}


\title{Claw-free cubic graphs and zero forcing}


\begin{abstract}
A \emph{claw-free cubic graph} is a cubic graph with no induced subgraph isomorphic to $K_{1,3}$. The zero forcing process begins with an initial set $S$ of colored vertices. At each step, a colored vertex with exactly one uncolored neighbor forces that neighbor to become colored. If repeated applications of this rule color every vertex of $G$, then $S$ is called a \emph{zero forcing set}. The minimum cardinality of a zero forcing set is the \emph{zero forcing number}, denoted by $Z(G)$. In this paper, we answer three open questions concerning upper bounds on the zero forcing number of claw-free cubic graphs in \cite{DH}.
    We characterize the connected claw-free cubic graphs satisfying $Z(G)=\alpha(G)+1$, where $\alpha(G)$ is the independence number. In addition, we establish the improved upper bound $Z(G)\leq \frac{T}{2}+D+2$ for claw-free cubic graphs with Hamiltonian contraction multigraphs, where $D$ is the number of diamonds and $T$ is the number of triangles in $G$.
\end{abstract}
\maketitle

\section{Introduction Known Results}
This section introduces the standard notation and definitions for the zero forcing and graph coloring processes, consistent with the framework established in \cite{DH}. 
Let $G = (V(G), E(G))$ be a graph with vertex set $V(G)$ and edge set $E(G)$. 
Let $S \subseteq V$ be an initial set of colored vertices, called tokens. 
The \emph{forcing process} proceeds in discrete time steps according to the following rule: a colored vertex $v$ \emph{forces} an uncolored vertex $u$ to be colored if $u$ is the only uncolored neighbor of $v$. 
If iteratively applying this rule colors all vertices in $V(G)$, then $S$ is called a \emph{zero forcing set} of $G$. 
The \emph{zero forcing number} of $G$, denoted by $Z(G)$, is the minimum cardinality of a zero forcing set. 
A \emph{forcing chain} $v_1 \rightarrow v_2 \rightarrow \dots \rightarrow v_k$ represents a sequence of chronological forces. 

Separately, a \emph{proper coloring} of $G$ assigns colors to the vertices of $G$ such that no two adjacent vertices share the same color. 
The minimum number of colors required for a proper coloring of $G$ is the \emph{chromatic number} of $G$, denoted by $\chi(G)$. 
The maximum degree of $G$ is denoted by $\Delta(G)$.

An \emph{$H$-free} graph is a graph that does not contain $H$ as an induced subgraph. 
In particular, a $K_{1,3}$-free graph is called a \emph{claw-free} graph.
A $3$-regular graph is referred to as a \emph{cubic} graph.
It was shown in \cite{HL} that the vertex set of any connected claw-free cubic graph can be uniquely partitioned into subsets that induce either a triangle unit ($K_3$) or a diamond unit of $G$ ($K_4\backslash e$). 
The total number of triangles and diamonds in $G$ are denoted by $T$ and $D$, respectively. 

For a graph $G$ on $n$ vertices, we let $\mathcal{S}(G)$ be the set of real symmetric $n\times n$ matrices $A=[a_{ij}]$ such that for $i\neq j$, $a_{ij}$ is nonzero if and only if $\{i,j\}\in E(G)$. The maximum nullity of $G$ is $M(G)=\max\{\text{null}(A); A\in \mathcal{S}(G)\}.$

The following theorem is used to provide a lower bound for the zero forcing. 

\begin{theorem}[\cite{MR2645093}]\label{mlessthanZ}
For any graph $G$, $M(G)\leq Z(G)$.
\end{theorem}

\begin{definition}[\textbf{Diamond Necklace} \cite{DH}]
For $m \geq 2$, a \emph{diamond necklace} $N_m$ is a connected, claw-free, cubic graph constructed from $m$ disjoint copies of diamond units $D_1, D_2, \dots, D_m$. Each unit $D_k$ has the vertex set $V(D_k) = \{a_k, b_k, c_k, d_k\}$, where $\{a_k, b_k\}$ is the unique missing edge. The units are linked cyclically by adding the edges $\{a_k, b_{k+1}\}$ for all $k \in [m]$, with indices taken modulo $m$. We denote the family of these graphs by $\mathcal{N}_{\text{cubic}} = \{ N_m \mid m \geq 2 \}$.  Both the zero forcing number and the independence number are known for diamond necklaces, and are given in the following lemma.
\end{definition}

\begin{lemma}\cite{DH}\label{Ncubic} If G $\in$ $\mathcal{N}_{cubic}$ with order $n$, then the following hold.  \end{lemma}

\begin{itemize}
\item $Z(G)=\frac{1}{4}n+2$
\item $\alpha(G)=\lfloor \displaystyle \frac{3}{8}n \rfloor $
\end{itemize}

\begin{definition}[\textbf{Diamond Bracelet} \cite{DH}]
For $m \geq 1$, a \emph{diamond bracelet} $B_m$ is obtained from a diamond necklace $N_{m+1}$ by replacing the diamond unit $D_{m+1}$ with a triangle unit $T_0$ where $V(T_0) = \{a_0, b_0, c_0\}$. The triangle is integrated into the graph by introducing the edges $\{b_0 ,b_1\}$ and $\{a_0, a_m\}$.
\end{definition}

\begin{definition}[\textbf{Diamond Chain} \cite{DH}]
For $m \geq 1$, a \emph{diamond chain} $L_m$ is obtained from a diamond necklace $N_{m+1}$ by replacing the diamond unit $D_{m+1}$ with two disjoint triangle units, $T_1$ and $T_2$. This structure is completed by adding an edge between $a_1$ and a vertex $t_1 \in V(T_1)$, and another edge between $b_m$ and a vertex $t_2 \in V(T_2)$.
\end{definition}

Figure~\ref{chain-bracelet} illustrates structural examples of these three graph families.

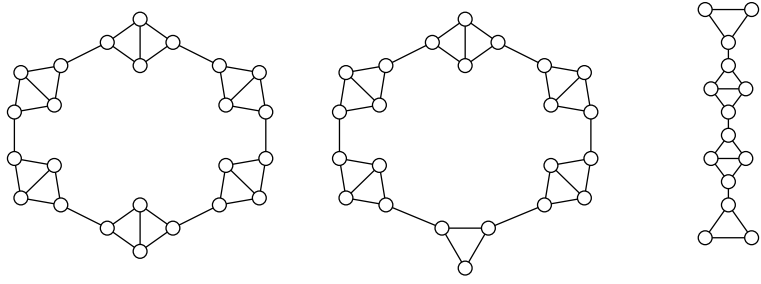
\begin{figure}[!h]
\begin{center}
\resizebox{0.8\textwidth}{!}{
\begin{tikzpicture}

\begin{scope} 

\node (l2) at (0.14,3.85) {};
\node (t2) at (1,4) {};
\node (r2) at (0.86,3.15) {};
\node (b2) at (0,3) {};
\draw (l2) --(t2)--(r2)--(b2)--(l2)--(r2);

\node (l1) at (0.14,1.15) {};
\node (b1) at (1,1) {};
\node (r1) at (0.86,1.85) {};
\node (t1) at (0,2) {};

\draw (l1)-- (t1)-- (r1) --(b1)--(l1) --(r1);

\node (t4) at (5.414,2) {};
\node (l4) at (4.554,1.85) {};
\node (r4) at (5.274,1.15) {};
\node (b4) at (4.414,1) {};

\draw (l4)--(t4)--(r4) --(b4)--(l4)--(r4);

\node (b3) at (5.414,3) {};
\node (l3) at (4.554,3.15) {};
\node (r3) at (5.274,3.85) {};
\node (t3) at (4.414,4) {};

\draw (l3)-- (t3)-- (r3) --(b3)--(l3) --(r3);

\node (l5) at (2,4.5) {};
\node (b5) at (2.707,4) {};
\node (t5) at (2.707,5) {};
\node (r5) at (3.414,4.5) {};

\draw (t5)--(r5)--(b5)--(l5)--(t5)--(b5);

\node (l6) at (2,0.5) {};
\node (b6) at (2.707,0) {};
\node (t6) at (2.707,1) {};
\node (r6) at (3.414,0.5) {};

\draw (t6)--(r6)--(b6)--(l6)--(t6)--(b6);

\draw (t1)--(b2);
\draw(t2)--(l5);
\draw (r5)--(t3);
\draw(b3)--(t4);

\draw(b4)--(r6);
\draw (l6)--(b1);

\end{scope}

\begin{scope} [xshift=7cm]
\node (l2) at (0.14,3.85) {};
\node (t2) at (1,4) {};
\node (r2) at (0.86,3.15) {};
\node (b2) at (0,3) {};
\draw (l2) --(t2)--(r2)--(b2)--(l2)--(r2);

\node (l1) at (0.14,1.15) {};
\node (b1) at (1,1) {};
\node (r1) at (0.86,1.85) {};
\node (t1) at (0,2) {};

\draw (l1)-- (t1)-- (r1) --(b1)--(l1) --(r1);

\node (t4) at (5.414,2) {};
\node (l4) at (4.554,1.85) {};
\node (r4) at (5.274,1.15) {};
\node (b4) at (4.414,1) {};

\draw (l4)--(t4)--(r4) --(b4)--(l4)--(r4);

\node (b3) at (5.414,3) {};
\node (l3) at (4.554,3.15) {};
\node (r3) at (5.274,3.85) {};
\node (t3) at (4.414,4) {};

\draw (l3)-- (t3)-- (r3) --(b3)--(l3) --(r3);

\node (l5) at (2,4.5) {};
\node (b5) at (2.707,4) {};
\node (t5) at (2.707,5) {};
\node (r5) at (3.414,4.5) {};

\draw (t5)--(r5)--(b5)--(l5)--(t5)--(b5);

\node (l) at(2.207,0.5){};
\node (r) at (3.207,0.5 ){};
\node (b) at (2.707,-0.36){};

\draw (l)--(r)--(b)--(l);

\draw (t1)--(b2);
\draw(t2)--(l5);
\draw (r5)--(t3);
\draw(b3)--(t4);
\draw(b4)--(r);
\draw (l)--(b1);


\end{scope}
\begin{scope} [xshift=15cm]

\node                      (tt1) at (0.375,1){};
\node                      (lt1) at (-0.125,0.293){};
\node                      (rt1) at (0.875,0.293){};
\draw (tt1) -- (lt1) -- (rt1) -- (tt1);

\node                      (bt2) at (0.375,4.5){};
\node                      (lt2) at (-0.125,5.207){};
\node                      (rt2) at (0.875,5.207){};
\draw (bt2) -- (lt2) -- (rt2) -- (bt2);

\node  (ld1) at (0,2) {};
\node  (td1) at (0.375,2.5) {};
\node  (rd1) at (0.75,2) {};
\node  (bd1) at (0.375,1.5) {};
\draw  (ld1) -- (td1) -- (rd1) -- (bd1) -- (ld1) -- (rd1);

\node  (ld2) at (0,3.5) {};
\node  (td2) at (0.375,4) {};
\node  (rd2) at (0.75,3.5) {};
\node  (bd2) at (0.375,3) {};
\draw  (ld2) -- (td2) -- (rd2) -- (bd2) -- (ld2) -- (rd2);

\draw (tt1)--(bd1);
\draw (td1)--(bd2);
\draw(td2)--(bt2);
\end{scope}

\end{tikzpicture}
}
\caption{Diamond necklace, diamond bracelet and diamond chain }
\label{chain-bracelet}

\end{center}\vspace{-15pt}
\end{figure} 
\bigskip

Davila and Henning~\cite{DH} introduced an algorithm to construct both an independent vertex set and a zero forcing set for connected, claw-free cubic graphs. 
This algorithm provides an alternative approach for advancing the zero forcing process when a forcing chain halts at either a triangle unit or a diamond unit. 
The execution of their algorithm relies on two localized reduction mechanisms for both of these cases: the \emph{triangle rule} and the \emph{diamond rule}, respectively. 

$\bold{Triangle\:rule}$ \cite{DH} After initiating the forcing game, suppose to be at a time step of the forcing process where a forcing chain halts.
We consider a triangle $T_k$ of $G$ with vertices  $V(T_k)=\{a_{k},b_{k},c_{k}\}$ where the only colored vertex is $a_{k}$.
Let $c_{k'}$ and $b_{k'}$ be the adjacent vertices of $c_k$ and $b_k$, respectively, where $c_{k'},b_{k'} \notin V(T_k)$. 
If neither $c_{k'}$ nor $b_{k'}$ are colored then we allow the forcing game to continue by coloring vertex $c_{k}$ and update $S$ as $S:=S\cup\{c_{k}\}$. We call this process, \emph{Case A } of the triangle rule.
If $c_{k'}$ is colored through its neighbor not in $T_k$, then $c_{k'}$ forces vertex $c_k$ and $c_{k}$ forces vertex $b_{k}$. Thus, all vertices of triangle $T_k$ become colored without updating the existing forcing set. We call this process, \emph{Case B } of the triangle rule.
Note all vertices of triangle $T_k$ end up colored regardless of which case is applied. 

$\bold{Diamond\:rule}$ \cite{DH}
Assume to be at a time step of the forcing process where a forcing chain halts.
Consider the diamond unit $D_k$ of $G$ with vertices $V(D_k)=\{a_{k},b_{k},c_{k},d_{k}\}$ where $\{a_k, b_k\}$ is the missing edge in $D_k$. 
Let $a_k$ be the only colored vertex in $D_k$, forced through its neighbor not in $D_k$. 
We allow the forcing game to proceed by updating the existing forcing set $S$ as $S:=S\cup\{ c_{k}\}$. Thus, $a_k$ forces $d_k$ and $d_k$ forces ${b_k}$.
All vertices of diamond $D_k$ end up colored.

Davila and Henning in \cite{DH} proved that for every connected claw-free cubic graph  $Z(G) \leq \alpha(G)+1$. The authors posted the following questions regarding this inequality. \\ 
Do the following inequalities hold for connected claw-free cubic graphs with sufficiently large number of vertices $n$?
\begin{enumerate}
\item\label{z<alpha} $Z(G) \leq \alpha(G)$ 
\item\label{z<3/4alpha} $Z(G)\leq \displaystyle \frac{3}{4}\alpha(G)$
\item\label{z<3/10n} $Z(G)\leq \displaystyle \frac{3}{10}n$. 
\end{enumerate}

He et al. answer the inequality \ref{z<alpha} affirmatively in \cite{He} by showing that the equation $Z(G) = \alpha(G)+1$ holds only for a finite set of connected claw-free cubic graphs. Additionally, they prove that the inequalities  \eqref{z<3/4alpha} and \eqref{z<3/10n} hold in the case where $G$ contains no diamond bracelets or parallel diamond chains. We provide families of graphs as counterexamples for these inequalities outside of this subset of graphs to show that they do not hold in general. 

\section{Main Results}

\subsection{Question 1}
One of the questions raised in \cite{DH}, is whether the following equation holds only for a finite set of connected claw-free cubic graphs.
\begin{equation}\label{eq:z=alpha+1}
Z(G)=\alpha(G)+1.
\end{equation} This question has been answered affirmatively by He et al. in \cite{He} by showing that graphs $C_{3}\boxempty K_{2}$, $N_{2}$ and $N_{3}$  in Figure~\ref{fig:z=alpha+1} are the only connected claw-free cubic graphs with $Z(G)= \alpha(G)+1$.  In this paper, we provide an alternative proof.

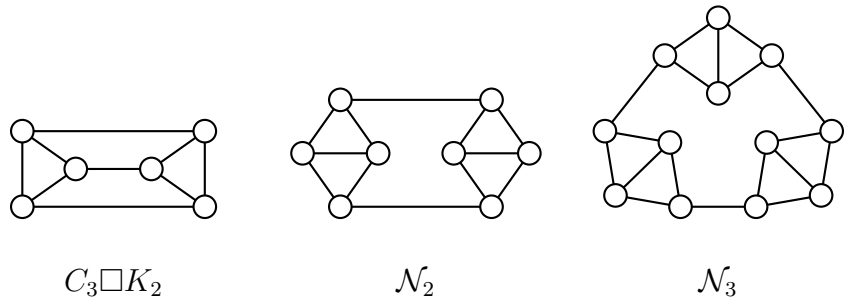
\begin{figure}[!h]
\begin{center}
\begin{tikzpicture}
\begin{scope} 
\node  (1) at (0.293,0.5) {};
\node  (2) at (1,1) {};
\node  (3) at (0.293,1.5) {};
\node (4) at (2,1) {};
\node (5) at (2.707,0.5) {};
\node (6) at (2.707,1.5) {};

\draw (1) -- (2) -- (3) -- (1);
\draw (4) -- (5) -- (6) -- (4);

\draw (3) -- (6);
\draw (2)--(4);
\draw (1) -- (5);

\node[rectangle,draw=none] at (1.5,-0.5) { $C_{3} \Box K_{2}$};
\end{scope}

\begin{scope} [xshift=4cm,yshift=0.5cm]
\node (1) at (0,0.707) {};
\node (2) at (0.5,0) {};
\node (3) at (1,0.707) {};
\node (4) at (0.5,1.414) {};

\node (5) at (2,0.707) {};
\node (6) at (2.5,0) {};
\node (7) at (3,0.707) {};
\node (8) at (2.5,1.414) {};

\draw (1) -- (2) -- (3) -- (4) -- (1) -- (3);
\draw (5) -- (6) -- (7) -- (8) -- (5) -- (7);

\draw (4)--(8);
\draw (2)--(6);
\node[rectangle,draw=none] at (1.5,-1) { $\mathcal{N}_{2}$};

\end{scope}

\begin{scope} [xshift=8cm,yshift=0.5cm]

\node (1) at (0.14,0.15) {};
\node (2) at (1,0) {};
\node (3) at (0.86,0.85) {};
\node (4) at (0,1) {};

\node (5) at (2.14,0.85) {};
\node (6) at (2,0) {};
\node (7) at (2.86,0.15) {};
\node (8) at (3,1) {};

\node (9) at (0.793,2) {};
\node (10) at (1.5,1.5) {};
\node (11) at (1.5,2.5) {};
\node (12) at (2.207,2) {};

\draw (1) -- (2) -- (3) -- (4) -- (1) -- (3);
\draw (5) -- (6) -- (7) -- (8) -- (5) -- (7);
\draw (10) -- (12) -- (11) -- (9) -- (10) -- (11);

\draw (4)--(9);
\draw (12)--(8);
\draw (2)--(6);
\node[rectangle,draw=none] at (1.5,-1) { $\mathcal{N}_{3}$};

\end{scope}

\end{tikzpicture}
\caption{Finite set of graphs that satisfy $ Z(G)=\alpha(G)+1$ }\label{fig:z=alpha+1}
\end{center}\vspace{-15pt}
\end{figure} 

Let $G=(V,E)$ be a connected claw-free cubic graph of order $n$. We consider the following three cases: 

\begin{enumerate}
\item $G$ is triangle-free
\item $G$ is diamond-free
\item $G$ contains both triangle and diamond units 
\end{enumerate}

\subsubsection{Graph $G$ is triangle-free.}

 Both, the zero forcing number and the independence number  are given in Lemma~\ref{Ncubic}. 
 For $n=8, 12$ we have $Z(G)=\alpha(G)+1$ and for $n=16, 20$ we have $Z(G)=\alpha(G)$. For all values with $n\geq 24$ we have $Z(G)<\alpha(G)$. Figure~\ref{fig:z=alpha+1} shows graphs ${N}_{2}$ and ${N}_{3}$ that satisfy $Z(G)=\alpha(G)+1$.

\subsubsection{$G$ is diamond-free.}

Since there must be at least two triangles for a connected cubic diamond-free claw-free graph, $G=C_{3} \Box K_{2}$, which is illustrated in Figure$~$\ref{fig:z=alpha+1}, is the graph in this family with the minimum number of triangles. Equation \eqref{eq:z=alpha+1} holds for $G=C_{3} \Box K_{2}$ since
 $Z(G)=3$ and $\alpha(G)=2$.
 For any other graph with more than two triangles Theorem~\ref{diamond-freeZleqalpha} shows that $Z(G)\leq \alpha(G)$. 

\begin{proposition}
Every connected claw-free cubic graph has an even number of triangle units.
\end{proposition}

\proof We have $\sum_{v\in V(G)} \deg(v)= 2 e$ where $e$ is the number of edges. Thus, $3n=2e$ which implies $n$ is even. Since every diamond unit has an even number of vertices, this means that the number of triangle units must be even. \qed 

\begin{remark}\label{Cubictriangle12}
The graphs illustrated on Figure \ref{diamond-free-on-4-triangles}, are the only connected claw-free cubic graphs on 12 vertices that are diamond-free, up to isomorphism.
\end{remark}

\begin{figure}[h!]
\begin{center}
\begin{tikzpicture}
\begin{scope} [xshift=4cm,yshift=0.5cm]
\node   [fill={blue}]   (1) at (0.293,1) {};
\node   [fill={blue}]   (2) at (1,1.5) {};
\node   [fill={blue}]   (3) at (0.293,2) {};
 
\node                      (4) at (2,1.5) {};
\node                      (5) at (2.707,1) {};
\node  [fill={blue}]    (6) at (2.707,2) {};

\node                      (7) at (1.5,1){};
\node                      (8) at (1,0.293){};
\node                      (9) at (2,0.293){};

\node                      (10) at (1.5,2){};
\node                      (11) at (1,2.707){};
\node                      (12) at (2,2.707){};

\draw (1) -- (2) -- (3) -- (1);
\draw (4) -- (5) -- (6) -- (4);
\draw (7) -- (8) -- (9) -- (7);
\draw (10) -- (11) -- (12) -- (10);

\draw (2)--(4);
\draw (3)--(11);
\draw (6) -- (12);
\draw (1) -- (8);
\draw (5) -- (9);
\draw (10) -- (7);


\end{scope}

\begin{scope} [xshift=8cm]

\node  [fill={blue}]     (7) at   (0.086,2.75) {};
\node  [fill={blue}]     (8) at   (0.793,2.25) {};
\node  [fill={blue}]     (9) at   (0.793,3.25) {};
\draw (7) -- (8) -- (9) -- (7);

\node                       (1) at   (1.543,2.25) {};
\node                       (2) at   (2.25,2.75) {};
\node                       (3) at   (1.543,3.25) {};
\draw  (1) -- (2) -- (3) -- (1);

\node                       (4) at   (0.086,1) {};
\node   [fill={blue}]    (5) at   (0.793,0.5) {};
\node                       (6) at   (0.793,1.5) {};
\draw (4) -- (5) -- (6) -- (4);

\node                       (10) at   (1.543,0.5) {};
\node                       (11) at   (2.25,1) {};
\node                       (12) at   (1.543,1.5) {};
\draw  (10) -- (11) -- (12) -- (10);

\draw (8) -- (1);
\draw (9) -- (3);
\draw (5) -- (10);
\draw (6) -- (12);
\draw (7) -- (4);
\draw (2) -- (11);

\end{scope}

\end{tikzpicture}
\caption{The only two connected claw-free cubic and diamond-free graph on 12 vertices with respective zero forcing sets}\label{diamond-free-on-4-triangles}
\end{center}\vspace{-15pt}
\end{figure}
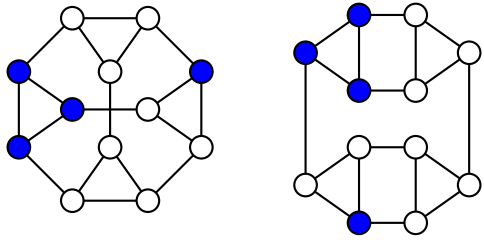

\begin{lemma}\label{triangleonly}
If $G$ is a connected claw-free cubic and diamond-free graph on $n$ vertices, then  $\alpha(G)=\displaystyle\frac{n}{3}$, i.e. the number of triangle units.
\end{lemma}

Using Brook's Theorem \cite{West}, since the graph $G$ is neither a complete graph nor a cycle on odd number of vertices, we have $\chi(G)\leq \Delta (G)$. On the other hand, for any graph $G$ on $n$ vertices, we have $\displaystyle \frac{n}{\alpha(G)}\leq \chi(G)$. Thus, 
\[\displaystyle \frac{n}{\alpha(G)}\leq \chi(G)\leq \Delta (G)\] This implies that $\displaystyle \frac{n}{\Delta(G)}\leq \alpha(G)$. Since $\Delta(G)=3$, we have   $\displaystyle \frac{n}{3}\leq \alpha(G)$. \\

Alternatively, if $G$ has an independent set $I$ that contains more than $\displaystyle \frac{n}{3}$ vertices, then there exists at least one triangle such that two of its vertices belong to $I$ which is a contradiction with the fact that $I$ is independent. This implies that $\alpha(G)\leq \frac{n}{3}$. Therefore, $\alpha(G)=\displaystyle\frac{n}{3}$. \qed \\ 

\begin{theorem}\label{diamond-freeZleqalpha}
For any connected claw-free cubic and diamond-free graph on at least 12 vertices, we have $Z(G)\leq \alpha(G)=\displaystyle \frac{n}{3}$.
\end{theorem}

\proof We use induction on the number of triangles. Using Remark \ref{Cubictriangle12}, up to isomorphism, the only two connected claw-free cubic graphs on 12 vertices are the graphs in Figure \ref{diamond-free-on-4-triangles}; in both cases $Z(G)=\alpha(G)=4$. Now suppose the inequality holds for all connected claw-free and diamond-free cubic graphs on at most $3k$ vertices where $k\geq 4$ is the the number of triangle units and is even. Let $G$ be such a graph on $3k+6$ vertices and consider $T_1, T_2$ with at least one edge connecting them, which are triangle units in $G$ with vertices $a_1,b_1,c_1$ and $a_2,b_2,c_2$, respectively.  There are two cases to consider; either $T_1$ and $T_2$ have two edges between them or they have one edge between them.  Note that we do not consider the case in which they have three edges between them, as the only case where a connected claw-free cubic graph will have this is if the graph is $C_3\square K_2$.  

\textbf{Case I:} Suppose that $T_1$ and $T_2$ are connected by two edges, namely $\{a_1,a_2\}$ and $\{b_1,b_2\}$. Let $\{c_1,d_1\}$ and $\{c_2,d_2\}$ denote the two edges of $G$ that connect $T_1$ and $T_2$, respectively, to the remainder of the graph. Note that $d_1\neq d_2$ otherwise $G$ contains a claw subgraph. Now let $G'$ be the graph with $T_1$ and $T_2$ removed and the edge $\{d_1, d_2\}$ added.  This is a connected claw-free cubic and diamond-free graph on $3k$ vertices as each vertex lies on a triangle unit, so by our inductive hypothesis we have $Z(G')\leq \alpha(G')$.

Now we assume that in $G'$ there is a forcing chain such that $d_1$ forces $d_2$.  We subdivide $\{d_1, d_2\}$ four times so that we create edges $\{d_1, c_1\}, \{c_1, b_1\}, \{b_1, b_2\}, \{b_2, c_2\}$ and $\{c_2, d_2\}$.  Then we add vertices $a_1$ and $a_2$ along with the edges $\{a_1, b_1\}, \{a_1, c_1\},\{a_2, b_2\}, \{a_2, c_2\}$ and $\{a_1, a_2\}$ to reconstruct our original graph $G$.  

Following the same forcing process as in $G'$ we have that $d_1$ forces $c_1$.  Then through \textit{Case A} of the triangle rule we force $T_1$ and through \textit{Case B} of the triangle rule we force $T_2$, which then in turn forces $d_2$.  This implies that $Z(G)\leq Z(G')+1\leq \alpha(G')+1$.  Additionally, by Lemma \ref{triangleonly} we have that $\alpha(G)=\alpha(G')+2$, so $Z(G)\leq \alpha(G)-1\leq \alpha(G)$.  

If this forcing chain does not exist in $G'$, then the all triangle units in $G$ other than $T_1$ and $T_2$ can be forced in the same manner as they are in $G'$ as the only difference between the two graphs is this edge.  Again by our inductive hypothesis we have that $Z(G')\leq \alpha(G')$.  From here we can color $T_1$ with \textit{Case A} of the triangle rule and $T_2$ with \textit{Case B} of the triangle rule.  Then we can obtain the same inequality $Z(G)\leq\alpha(G)-1\leq \alpha(G)$.

\textbf{Case II: } Now suppose there is only one edge connecting $T_1$ and $T_2$.  We now have two subcases.  The first is that there exists a triangle $T_3$ such that either $T_1$ or $T_2$ shares two edges with, say $T_2$.  Then we can instead take $T_2$ and $T_3$ and use the proof for the previous case.

Otherwise, $T_1$ and $T_2$ connect through the edge $\{a_1, a_2\}$ and the other edges incident to these triangles are  $\{b_1,e_1\}$,$\{c_1,d_1\}$,$\{b_2,e_2\},$ and $\{c_2,d_2\}$.  We note that in this subcase $\{d_1, e_1\}, \{d_2, e_2\}\notin E(G)$, otherwise $G$ contains an induced claw subgraph.  Then let $G'$ be the graph obtained by deleting $T_1$ and $T_2$ and adding edges $\{d_1, e_1\}$ and $\{d_2, e_2\}$.  The graph $G'$ is a connected claw-free cubic and diamond-free graph on $3k$ vertices as each vertex lies in a triangle unit, so by the inductive hypothesis $Z(G')\leq \alpha(G')$.

 Now suppose that in the forcing process of $G'$ there is a forcing chain such that $e_1$ forces $d_1$ and $e_2$ forces $d_2$.  Now we subdivide the edge $\{d_1, e_1\}$ twice to obtain the edges $\{e_1, b_1\}, \{b_1, c_1\}, \{c_1, d_1\}$ and we do the same to the edge $\{d_2, e_2\}$ to obtain the edges $\{e_2, b_2\}, \{b_2, c_2\}, \{c_2, d_2\}$.  Now add the vertices $a_1$ and $a_2$ along with the edges $\{a_1, b_1\},$ $\{a_1, c_1\},$ $\{a_2, b_2\}, \{a_2, c_2\}$ and $\{a_1, a_2\}$ to obtain our original graph $G$.  Then $e_1$ forces $b_1$ and $e_2$ forces $b_2$ following the same forcing process as $G'$.  Then we can use \textit{Case A} of triangle rule followed by the \textit{Case B} of the triangle rule to fill triangles $T_1$ and $T_2$.  Then the vertices $d_1$ and $d_2$ will be forced.  This implies that $Z(G)\leq Z(G')+1\leq \alpha(G')+1$.  Additionally, by Lemma \ref{triangleonly} we have that $\alpha(G)=\alpha(G')+2$, so $Z(G)\leq \alpha(G)-1\leq \alpha(G)$. 

In the case where only one of these forcing chains exists in the forcing process of $G'$, say $e_1$ forces $d_1$, then after performing the same subdivisions and vertex additions to obtain $G$ we have that $e_1$ will force $b_1$ after performing the same forcing process as $G'$.  Then we can use \textit{Case A} of the triangle rule to fill $T_1$ in its entirety.  This will in turn force $d_1$ in addition to $a_2$.  Then \textit{Case A} of the triangle rule can be applied again for $T_2$, filling it completely.  The remainder of the forcing process proceeds as it did for $G'$.  This implies that $Z(G)\leq Z(G')+2\leq \alpha(G')+2$.  Additionally, by Lemma \ref{triangleonly} $\alpha(G)=\alpha(G')+2$, so $Z(G)\leq \alpha(G)$

If neither of these forcing chains exist, then the forcing process on $G$ proceeds the same as $G'$ until $T_1$ and $T_2$ are reached.  Then we can apply \textit{Case A} of the triangle rule to $T_1$ and \textit{Case B} of the triangle rule to $T_2$ to fill the remainder of the graph.  This implies that $Z(G)\leq Z(G')+1\leq \alpha(G')+1$.  Using Lemma \ref{triangleonly} we have $\alpha(G)=\alpha(G')+2$, so $Z(G)\leq \alpha(G)-1\leq \alpha(G)$.
 \qed

\subsubsection{Graph $G$ has both triangle and diamond units.} 
We consider two cases:  Case I: every diamond in $G$ belongs to a diamond chain; Case II: there is at least one diamond in $G$ that belongs to a diamond bracelet.\\
\\\textbf{Case I}:  
We obtain the graph $G'$ by deleting all diamond units in every chain and replacing them with an edge in each chain.
As a result  $G'$ is a connected claw-free cubic graph where every vertex belongs to an induced triangle unit. This implies that if $\mid V(G) \mid \geq 12$ then by Theorem \ref{diamond-freeZleqalpha} the inequality $Z(G') \leq \alpha(G') $ holds. 
We then replace every previously deleted diamond unit back into $G$ and update the zero forcing and the independent sets by applying the diamond rule which preserves our inequality $\mid S \mid \leq \mid I \mid$ and thus, $Z(G) \leq \alpha(G)$.

\textbf{Case II}:  
Let $T_1$ with $V(T_{1})=\{a,b,c\}$ be the triangle unit in a diamond bracelet $B_k$ and $D_i$ with $V(D_{i})=\{a_{i},b_{i},c_{i},d_{i}\}$ where the missing edge is $\{a_{i},b_{i}\} $ be the $i$'th  diamond where $D_1$ is connected to $T_1$ by the edge $\{b_{1},b\}$ and $D_{k}$ is connected to $T_1$ by the edge $\{a_{k},a\}$.  
If $G$ has more than one diamond bracelet we then initialize the forcing game in any arbitrary chosen diamond bracelet $B_{k}$ of $G$ as illustrated in Figure \ref {fig:Bk}. Initially $S =\{a_{1},d_{1},b_{1}\}$ and $I=\{a_{1},b_{1},c\}$.
We update $S$ and $I$ by considering a unit per time step.
If such unit is a triangle then we apply the triangle rule. If such unit is a diamond then we apply the diamond rule.
Note vertex $d_{1}$ can force vertex $c_1$ and $b_1$ forces $b$. 
Also, assume to be at time step where diamond $D_k$ was already colored then vertex $a_k$ forces $a$ and $a$ forces $c$. 
Thus, all vertices in triangle $T_1$ and diamond $D_1$ become colored without updating the existing forcing set.
Let $D$ be the number of diamond units beside $D_{1}$ that belong to $G$ and $T$ be the number of triangle units beside $T_1$ that belong to $G$. We note that once we have colored this initial bracelet that every remaining unit can be colored using either the triangle or diamond rule.  Thus, $\mid S \mid= 3+D+T$ and $\mid I \mid= 3+D+T$ which implies our inequality $Z(G)\leq \alpha(G)$ holds.  

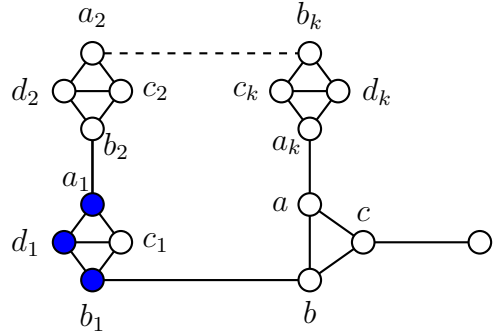
\begin{figure}[h!]
\begin{center}
\begin{tikzpicture}

\begin{scope}

\node [fill={blue}, label=left:{$d_{1}$}] (1) at (0,1) {};
\node  [fill={blue},label={[label distance=0cm]110:$a_{1}$}] (2) at (0.375,1.5) {};
\node [label=right:{$c_{1}$}] (3) at (0.75,1) {};
\node [fill={blue},label=below:{$b_{1}$}] (4) at (0.375,0.5) {};
\draw  (1) -- (2) -- (3) -- (4) -- (1) -- (3);

\node  [label=left:{$a$}](5) at (3.25,1.5) {};
\node  [label={$c$}] (6) at (3.957,1) {};
\node  [label=below:{$b$}](7) at (3.25,0.5) {};
\draw (5) -- (6) -- (7) -- (5) ;

\node [label=left:{$d_{2}$}] (10) at (0,3) {};
\node  [label={$a_{2}$}] (20) at (0.375,3.5) {};
\node [label=right:{$c_{2}$}] (30) at (0.75,3) {};
\node [label={[label distance=0cm]-20:$b_{2}$}] (40) at (0.375,2.5) {};
\draw  (10) -- (20) -- (30) -- (40) -- (10) -- (30);

\node [label=left:{$c_{k}$}] (01) at (2.875,3) {};
\node  [label={$b_{k}$}] (02) at (3.25,3.5) {};
\node [label=right:{$d_{k}$}] (03) at (3.625,3) {};
\node [label={[label distance=0cm]200:$a_{k}$}] (04) at (3.25,2.5) {};
\draw  (01) -- (02) -- (03) -- (04) -- (01) -- (03);

\node (v) at (5.5,1) {} ;

\draw (2) -- (40);
\draw (4) -- (7);

\draw [dashed] (20) -- (02); 
\draw (2) -- (40);
\draw (04) -- (5);
\draw (6) -- (v);
\end{scope}
\end{tikzpicture}
\caption{Initialization is made in $B_{k}$, where $k$ is the number of dimond units.}
\label{fig:Bk}

\end{center}\vspace{-15pt}
\end{figure}


\subsection{Questions 2 and 3}
Dalvia and Henning additionally, pose the question of whether or not the following inequalities hold for connected claw-free cubic graph of sufficiently large order $n$ \cite{DH}.
\begin{enumerate}
\setcounter{enumi}{1}
\item $Z(G)\leq\frac{3}{4}\alpha(G)$
\item $Z(G)\leq\frac{3}{10}n$
\end{enumerate}
He et al. (2024) answered these questions affirmatively when $G$ is a connected claw-free cubic graph with at least one triangle and without diamond bracelets or parallel diamond chains \cite{He}.  However this does not prove these inequalities in general.  We answer both inequalities negatively by providing families of connected claw-free cubic graphs for each case as a counter-example.

\subsubsection{$Z(G)\protect\leq\frac{3}{4}\alpha(G)$}

To disprove inequality \eqref{z<3/4alpha}, we construct an infinite family of connected claw-free cubic graphs whose zero forcing numbers exceed $\frac{3}{4}\alpha(G)$. We begin with the graph $F_1$, shown in Figure~\ref{fig:F1}, together with a zero forcing set of size $5$.

To show that this zero forcing set is minimum, we construct a matrix $A\in\mathcal{S}(F_1)$ with $\operatorname{rank}(A)=9$. The matrix $A$ is shown in Figure~\ref{fig:matrix}, and a basis for its nullspace is given in Figure~\ref{nullvec}. Since
\[
|V(F_1)|-\operatorname{rank}(A)=14-9=5,
\]
Theorem~\ref{mlessthanZ} implies that $5\le Z(F_1)$. Because we have already exhibited a zero forcing set of size $5$, it follows that
\[
Z(F_1)=5.
\]

\begin{lemma} \label{lem:F1}
    For the graph $F_1$ we have $Z(F_1)=5$.
\end{lemma}

\begin{figure}[h!] 
\begin{center}
\resizebox{0.5\textwidth}{!}{
\begin{tikzpicture}[label distance =0]
\begin{scope}

\node  [label=left:{$v^{6}$},fill={blue}]                              (1) at (-0.75,3.414) {};
\node  [label=above:{$v^{7}$},fill={blue}]                          (2) at (-0.375,3.914) {};
\node  [label={[label distance=0cm]200:$v^{5}$},fill={blue}] (3) at (-0.375,2.914) {};
\node  [label=right:{$v^{13}$}]                                           (4) at (0,3.414) {};
\draw  (1) -- (2) -- (4) -- (3) -- (1) -- (4);

\node  [label={[label distance=0cm]200:$v^{4}$}]                (5) at (0.25,2.414) {};
\node  [label=below:{$v^{12}$}]                                         (6) at (1.25,2.414) {};
\node  [label={[label distance=0cm]200:$v^{3}$}]                (7) at (0.75,1.707) {};
\draw  (5) -- (6) -- (7) -- (5) ;

\node  [label=below:{$v^{2}$}]                                         (20) at (0.75,0.957) {};
\node  [label=below:{$v^{1}$}]                                         (21) at (0.25,0.25) {};
\node  [label=below:{$v^{11}$},fill={blue}]                           (22) at (1.25,0.25) {};
\draw  (20) -- (21) -- (22) -- (20) ;

\node  [label=left:{$v^{14}$},fill={blue}]                               (11) at (1.5,3.414) {};
\node  [label=above:{$v^{8}$}]                                          (12) at (1.875,3.914) {};
\node  [label=below:{$v^{10}$}]                                          (13) at (1.875,2.914) {};
\node  [label=right:{$v^{9}$}]                                            (14) at (2.25,3.414) {};
\draw  (11) -- (12) -- (14) -- (13) -- (11) -- (14);

\draw  (2) -- (12);
\draw  (3) -- (5);
\draw  (13)--(6);
\draw  (7) -- (20);

\end{scope}

\end{tikzpicture}
}
\caption{Graph $F_1$ with a zero forcing set of size $5$}
\label{fig:F1}
\end{center}
\vspace{-15pt}
\end{figure}
\newpage
\setcounter{MaxMatrixCols}{15}
\begin{figure}
    \centering
    $$A=\begin{bmatrix}
 2 & -1 & 0 & 0 & 0 & 0 & 0 & 0 & 0 & 0 & 2 & 0 & 0 & 0\\
 -1 & 2 & -1 & 0 & 0 & 0 & 0 & 0 & 0 & 0 & -1 & 0 & 0 & 0\\
 0 & -1 & 2 & -1 & 0 & 0 & 0 & 0 & 0 & 0 & 0 & -1 & 0 & 0\\
 0 & 0 & -1 & 2 & -1 & 0 & 0 & 0 & 0 & 0 & 0 & 1 & 0 & 0\\
 0 & 0 & 0 & -1 & 2 & -1 & 0 & 0 & 0 & 0 & 0 & 0 & -1 & 0\\
 0 & 0 & 0 & 0 & -1 & 2 & -1 & 0 & 0 & 0 & 0 & 0 & 2 & 0\\
 0 & 0 & 0 & 0 & 0 & -1 & 2 & -1 & 0 & 0 & 0 & 0 & -1 & 0\\
 0 & 0 & 0 & 0 & 0 & 0 & -1 & 2 & -1 & 0 & 0 & 0 & 0 & -1\\
 0 & 0 & 0 & 0 & 0 & 0 & 0 & -1 & 2 & -1 & 0 & 0 & 0 & 2\\
 0 & 0 & 0 & 0 & 0 & 0 & 0 & 0 & -1 & 0.9 & 0 & -0.1 & 0 & -1\\
 2 & -1 & 0 & 0 & 0 & 0 & 0 & 0 & 0 & 0 & 2 & 0 & 0 & 0\\
 0 & 0 & -1 & 1 & 0 & 0 & 0 & 0 & 0 & -0.1 & 0 & 0.9 & 0 & 0\\
 0 & 0 & 0 & 0 & -1 & 2 & -1 & 0 & 0 & 0 & 0 & 0 & 2 & 0\\
 0 & 0 & 0 & 0 & 0 & 0 & 0 & -1 & 2 & -1 & 0 & 0 & 0 & 2
\end{bmatrix}$$
    \caption{$A\in S(F_1)$ such that $\text{rank}(A)=9$}
    \label{fig:matrix}
\end{figure}

\begin{figure}
    \centering
    $\text{Null}(A)=\text{span}\left\{\begin{bmatrix}
1\\0\\0\\0\\0\\0\\0\\0\\0\\0\\-1\\0\\0\\0
    \end{bmatrix}, 
    \begin{bmatrix}
        0\\1\\\frac{3}{2}\\0\\\frac{1}{2}\\0\\\frac{3}{2}\\2\\0\\3\\\frac{1}{2}\\2\\1\\\frac{5}{2}
    \end{bmatrix},
    \begin{bmatrix}
        0 \\0\\0\\1\\1\\0\\1\\1\\0\\1\\0\\-1\\1\\1
    \end{bmatrix},
 \begin{bmatrix}
        0\\0\\0\\0\\0\\1\\0\\0\\0\\0\\0\\0\\-1\\0
    \end{bmatrix},
\begin{bmatrix}
        0\\0\\0\\0\\0\\0\\0\\0\\1\\0\\0\\0\\0\\-1
    \end{bmatrix}
    \right\}$
    \caption{Five linearly independent null-vectors of $A$}
    \label{nullvec}
\end{figure}

We now define $G_{k}$ as the graph obtained by $k\geq2$ vertex disjoint copies of $F_1$, where the vertices of degree two of the triangles are connected with each other to form a cycle $C_{2k}$.
Thus  $\mid V(G_{k})\mid=14k$.
Figure~\ref{fig:Gk} illustrates the graph $G_k$, where the colored vertices form a zero forcing set of $G_{k}$ of size $4k$, implying $Z(G_k)\leq 4k$. Additionally, $I=\{v_i^2, v_i^4, v_i^6, v_i^8, v_i^{10}:1\leq i\leq k\}$ is an independent set of size $5k$ of $G_k$.  We remark that this is a maximum independent set on $G_k$  because each triangle can have at most one vertex in an independent set.  Additionally each 14-vertex block is isomorphic to $F_1$, and $\alpha(F_1)=5$. Indeed, an independent set of size $6$ would require selecting two vertices from each diamond, but this is impossible because vertices $v_i^{7}$ and $v_i^{8}$ are adjacent. Therefore,
\[
\alpha(G_k)\le 5k.
\]

\begin{figure}[h!] 
\begin{center}
\resizebox{1\textwidth}{!}{
\begin{tikzpicture}[label distance =0]
\begin{scope}

\node  [label=left:{$v_{1}^{6}$},fill={blue}]                              (1) at (-0.75,3.414) {};
\node  [label=above:{$v_{1}^{7}$},fill={blue}]                          (2) at (-0.375,3.914) {};
\node  [label={[label distance=0cm]200:$v_{1}^{5}$},fill={blue}] (3) at (-0.375,2.914) {};
\node  [label=right:{$v_{1}^{13}$}]                                           (4) at (0,3.414) {};
\draw  (1) -- (2) -- (4) -- (3) -- (1) -- (4);

\node  [label={[label distance=0cm]200:$v_{1}^{4}$}]                (5) at (0.25,2.414) {};
\node  [label=below:{$v_{1}^{12}$}]                                         (6) at (1.25,2.414) {};
\node  [label={[label distance=0cm]200:$v_{1}^{3}$}]                (7) at (0.75,1.707) {};
\draw  (5) -- (6) -- (7) -- (5) ;

\node  [label=below:{$v_{1}^{2}$}]                                         (20) at (0.75,0.957) {};
\node  [label=below:{$v_{1}^{1}$}]                                         (21) at (0.25,0.25) {};
\node  [label=below:{$v_{1}^{11}$},fill={blue}]                           (22) at (1.25,0.25) {};
\draw  (20) -- (21) -- (22) -- (20) ;

\node  [label=left:{$v_{1}^{14}$},fill={blue}]                               (11) at (1.5,3.414) {};
\node  [label=above:{$v_{1}^{8}$}]                                          (12) at (1.875,3.914) {};
\node  [label=below:{$v_{1}^{10}$}]                                          (13) at (1.875,2.914) {};
\node  [label=right:{$v_{1}^{9}$}]                                            (14) at (2.25,3.414) {};
\draw  (11) -- (12) -- (14) -- (13) -- (11) -- (14);

\draw  (2) -- (12);
\draw  (3) -- (5);
\draw  (13)--(6);
\draw  (7) -- (20);


\node [label=left:{$v_{2}^{6}$},fill={blue}] (1) at (3.75,3.414) {};
\node  [label=above:{$v_{2}^{7}$}](2) at (4.125,3.914) {};
\node  [label={[label distance=0cm]200:$v_{2}^{5}$},](3) at (4.125,2.914) {};
\node  [label=right:{$v_{2}^{13}$}] (4) at (4.5,3.414) {};
\draw (1) -- (2) -- (4) -- (3) -- (1) -- (4);

\node  [label={[label distance=0cm]200:$v_{2}^{4}$}] (5) at (4.75,2.414) {};
\node [ label=below:{$v_{2}^{12}$},fill={blue}] (6) at (5.75,2.414) {};
\node  [label={[label distance=0cm]200:$v_{2}^{3}$}](7) at (5.25,1.707) {};
\draw (5) -- (6) -- (7) -- (5) ;

\node  [label=below:{$v_{2}^{2}$}] (23) at (5.25,0.957) {};
\node  [label=below:{$v_{2}^{1}$}] (24) at (4.75,0.25) {};
\node  [label=below:{$v_{2}^{11}$}, fill = {blue}](25) at (5.75,0.25) {};
\draw (23) -- (24) -- (25) -- (23) ;

\node [label=left:{$v_{2}^{14}$},fill={blue}] (11) at (6,3.414) {};
\node [label=above:{$v_{2}^{8}$}] (12) at (6.375,3.914) {};
\node   [label=below:{$v_{2}^{10}$}] (13) at (6.375,2.914) {};
\node   [label=right:{$v_{2}^{9}$}] (14) at (6.75,3.414) {};
\draw (11) -- (12) -- (14) -- (13) -- (11) -- (14);

\draw (2) -- (12);
\draw (3) -- (5);
\draw (13)--(6);
\draw (7) -- (23);


\node  [label=left:{$v_{3}^{6}$}, fill={blue}] (1) at (8.50,3.414) {};
\node  [label=above:{$v_{3}^{7}$}] (2) at (8.875,3.914) {};
\node  [label={[label distance=0cm]200:$v_{3}^{5}$},] (3) at (8.875,2.914) {};
\node  [label=right:{$v_{3}^{13}$}] (4) at (9.25,3.414) {};
\draw (1) -- (2) -- (4) -- (3) -- (1) -- (4);

\node  [label={[label distance=0cm]200:$v_{3}^{4}$},] (5) at (9.5,2.414) {};
\node   [ label=below:{$v_{3}^{12}$},fill={blue}] (6) at (10.5,2.414) {};
\node  [label={[label distance=0cm]200:$v_{3}^{3}$}] (7) at (10,1.707) {};
\draw (5) -- (6) -- (7) -- (5) ;

\node   [label=below:{$v_{3}^{2}$}]  (26) at (10,0.957) {};
\node  [label=below:{$v_{3}^{1}$}] (27) at (9.5,0.25) {};
\node  [label=below:{$v_{3}^{11}$},fill={blue}](28) at (10.5,0.25) {};
\draw (26) -- (27) -- (28) -- (26) ;

\node [label=left:{$v_{3}^{14}$},fill={blue}] (11) at (10.75,3.414) {};
\node [label=above:{$v_{3}^{8}$}]   (12) at (11.125,3.914) {};
\node [label=below:{$v_{3}^{10}$}] (13) at (11.125,2.914) {};
\node [label=right:{$v_{3}^9$}] (14) at (11.5,3.414) {};
\draw (11) -- (12) -- (14) -- (13) -- (11) -- (14);

\draw (2) -- (12);
\draw (3) -- (5);
\draw (13)--(6);
\draw (7) -- (26);


\node [label=left:{$v_{k}^{6}$}, fill={blue}]  (1) at (13.25,3.414) {};
\node  [label=above:{$v_{k}^{7}$}]  (2) at (13.625,3.914) {};
\node  [label={[label distance=0cm]200:$v_{k}^{5}$}](3) at (13.625,2.914) {};
\node  [label=right:{$v_{k}^{13}$}] (4) at (14,3.414) {};
\draw (1) -- (2) -- (4) -- (3) -- (1) -- (4);

\node  [label={[label distance=0cm]200:$v_{k}^{4}$}]  (5) at (14.25,2.414) {};
\node  [label=below:{$v_{k}^{12}$},fill={blue}](6) at (15.25,2.414) {};
\node  [label={[label distance=0cm]200:$v_{k}^{3}$}]  (7) at (14.75,1.707) {};
\draw (5) -- (6) -- (7) -- (5) ;

\node   [label=below:{$v_{k}^{2}$}]  (29) at (14.75,0.957) {};
\node  [label=below:{$v_{k}^{1}$}]  (30) at (14.25,0.25) {};
\node [label=below:{$v_{k}^{11}$}] (31) at (15.25,0.25) {};
\draw (29) -- (30) -- (31) -- (29) ;

\node [label=left:{$v_{k}^{14}$},fill={blue}]  (11) at (15.5,3.414) {};
\node [label=above:{$v_{k}^{8}$}]  (12) at (15.875,3.914) {};
\node [label=below:{$v_{k}^{10}$}] (13) at (15.875,2.914) {};
\node [label=right:{$v_{k}^{9}$}] (14) at (16.25,3.414) {};
\draw (11) -- (12) -- (14) -- (13) -- (11) -- (14);

\draw (2) -- (12);
\draw (3) -- (5);
\draw (13)--(6);
\draw (7) -- (29);

\draw (22)--(24);
\draw (25) -- (27);
\draw[dashed] (28) -- (30);

\draw  (21)-- (0,0.25); 
\draw  (0,0.25) -- (0,-0.5);
\draw  (0,-0.5)-- (15.5,-0.5);

\draw (31) -- (15.5,0.25);
\draw (15.5,0.25) --(15.5,-0.5);

\end{scope}

\end{tikzpicture}
}
\caption{A zero forcing set of $G_k$ with size $4k$}
\label{fig:Gk}
\end{center}
\vspace{-15pt}
\end{figure} 

Since Figure~\ref{fig:Gk} shows a zero forcing set of size $4k$, it remains only to show that $4k\le Z(G_k)$. To establish this lower bound, we use a known result on the behavior of the zero forcing number under edge deletion. Specifically, we apply the following theorem.

\begin{lemma} \label{deletion} \cite{Edge}
    Let $e$ be any edge of the graph $G$, then $-1\leq Z(G)-Z(G-e)\leq 1$.
\end{lemma}

\begin{lemma} \label{lem:Gk}
    For the graph $G_K$ we have $Z(G_k)=4k$.
\end{lemma}

\begin{proof}

Consider the graph $kF_1$ consisiting of $k$ disjoint copies of $F_1$. By Lemma \ref{lem:F1} $Z(F_1)=5$, and hence $Z(kF_1) = 5k$.  Furthermore, $kF_1$ can be obtained from $G_k$ by deleting the $k$ edges of the main cycle. Therefore, Lemma~\ref{deletion} implies that $-k\leq Z(G_k)-Z(kF_1)\leq k$. Substituting $Z(kF_1)=5k$ yields $4k\leq Z(G_k)\leq 6k$. Since we have already constructed a zero forcing set of size $4k$, it follows that $Z(G_k)=4k$. Finally $\alpha(G_k)=5k$ and therefore $Z(G_k) = \frac{4}{5} \alpha (G_k)$.
\end{proof}

This shows that there is greater boundary than the one proposed in \cite{DH}.  

\subsubsection{$Z(G)\leq\frac{3}{10}n$}
We similarly construct a family of connected claw-free cubic graphs that does not satisfy the bound in inequality \eqref{z<3/10n}. We first construct the graph $F_2$, shown in Figure \ref{fig:F2}, together with a zero forcing set.

\begin{lemma} \label{lem:F2}
   For the graph $F_2$ we have $Z(F_2)=3$.
\end{lemma}

\begin{proof}
The minimum ranks for all graphs on at most 7 vertices have been provided in \cite{MR}, and $\text{mr}(F_2)=4$.  This implies that $|V(F_2)|-\text{mr}(F_2)=7-4=3\leq Z(F_2)$ giving us our lower bound.  Additionally, the zero forcing set of size $3$ in Figure \ref{fig:F2} which implies $Z(F_2)=3$.
\end{proof}

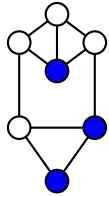
\begin{figure}[h!]
    \centering
    \begin{tikzpicture}

\node                      (11) at (2.707,3.164) {};
\node   [fill={blue}]  (12) at (3.707,3.164) {};
\node [fill={blue}]                    (13) at (3.207,2.457) {};
\draw  (11) -- (12) -- (13) -- (11) ;

\node  (14) at (2.707,4.289) {};
\node  (15) at (3.207,4.664) {};
\node  (16) at (3.707,4.289) {};
\node  [fill={blue}] (17) at (3.207,3.914) {};
\draw  (15) -- (16) -- (17) -- (14) -- (15) -- (17);

\draw (11)--(14);
\draw (12)--(16);
    \end{tikzpicture}
    \caption{Graph $F_2$ with a zero forcing set of size 3}
    \label{fig:F2}
\end{figure}

We now construct a connected claw-free cubic graph $H_k$ by connecting $k+2$ copies of $F_2$ to the degree two vertices of a triangle chain of length $k$.
Thus $\mid V(H_{k})\mid = n = 14+10 k$.
Figure \ref{fig:Hk} illustrates $H_k$ where the colored vertices provide a zero forcing set.\\
\\

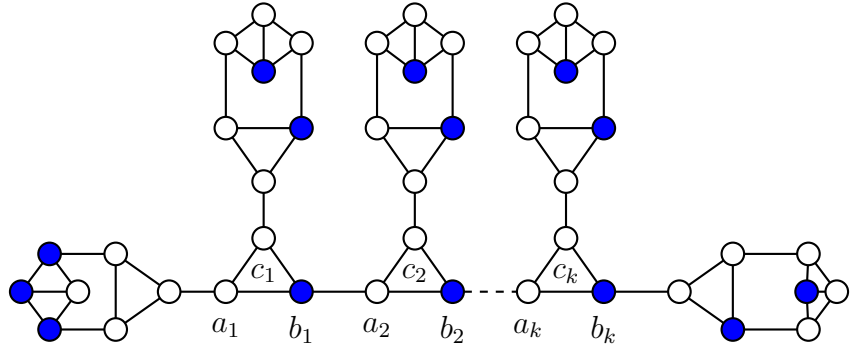
\begin{figure}[h!] 
\begin{center}
\begin{tikzpicture}

\begin{scope}
\node  [fill={blue}] (1) at (0,1) {};
\node  [fill={blue}] (2) at (0.375,1.5) {};
\node  (3) at (0.75,1) {};
\node  [fill={blue}]  (4) at (0.375,0.5) {};
\draw  (1) -- (2) -- (3) -- (4) -- (1) -- (3);

\node  (5) at (1.25,1.5) {};
\node  (6) at (1.957,1) {};
\node  (7) at (1.25,0.5) {};
\draw  (5) -- (6) -- (7) -- (5) ;

\draw  (2) -- (5);
\draw  (4) -- (7);

\node    [label=below:{$a_{1}$}]                   (8) at (2.707,1) {};
\node    [label=below:{$c_{1}$}]                   (9) at (3.207,1.707) {};
\node    [label=below:{$b_{1}$},fill={blue}]    (10) at (3.707,01) {};
\draw  (8) -- (9) -- (10) -- (8) ;

\node                     (11) at (2.707,3.164) {};
\node   [fill={blue}]  (12) at (3.707,3.164) {};
\node                     (13) at (3.207,2.457) {};
\draw  (11) -- (12) -- (13) -- (11) ;

\node  (14) at (2.707,4.289) {};
\node  (15) at (3.207,4.664) {};
\node  (16) at (3.707,4.289) {};
\node  [fill={blue}] (17) at (3.207,3.914) {};
\draw  (15) -- (16) -- (17) -- (14) -- (15) -- (17);

\draw (6)--(8);
\draw (9)--(13);
\draw (11)--(14);
\draw (12)--(16);

\node     [label=below:{$a_{2}$}]                   (08) at (4.707,1) {};
\node     [label=below:{$c_{2}$}]                   (09) at (5.207,1.707) {};
\node     [label=below:{$b_{2}$},fill={blue}]    (010) at (5.707,01) {};
\draw     (08) -- (09) -- (010) -- (08) ;

\node                       (011) at (4.707,3.164) {};
\node     [fill={blue}]  (012) at (5.707,3.164) {};
\node                       (013) at (5.207,2.457) {};
\draw    (011) -- (012) -- (013) -- (011) ;

\node  (014) at (4.707,4.289) {};
\node  (015) at (5.207,4.664) {};
\node  (016) at (5.707,4.289) {};
\node  [fill={blue}] (017) at (5.207,3.914) {};
\draw  (015) -- (016) -- (017) -- (014) -- (015) -- (017);

\draw (10) --(08);
\draw (09)--(013);
\draw (014)--(011);
\draw (016)--(012);

\node    [label=below:{$a_{k}$}]                   (008) at (6.707,1) {};
\node    [label=below:{$c_{k}$}]                   (009) at (7.207,1.707) {};
\node    [label=below:{$b_{k}$},fill={blue}]     (0010) at (7.707,01) {};
\draw    (008) -- (009) -- (0010) -- (008) ;

\node                     (0011) at (6.707,3.164) {};
\node     [fill=blue]   (0012) at (7.707,3.164) {};
\node                     (0013) at (7.207,2.457) {};
\draw (0011) -- (0012) -- (0013) -- (0011) ;

\node  (0014) at (6.707,4.289) {};
\node  (0015) at (7.207,4.664) {};
\node  (0016) at (7.707,4.289) {};
\node [fill={blue}]  (0017) at (7.207,3.914) {};
\draw  (0015) -- (0016) -- (0017) -- (0014) -- (0015) -- (0017);

\draw [dashed](010) -- (008);
\draw (009)--(0013);
\draw (0011)--(0014);
\draw  (0012)--(0016);

\node  (918) at (8.707,1) {};
\node  (919) at (9.414,1.5) {};
\node  [fill={blue}] (920) at (9.414,0.5) {};
\draw (918) -- (919) -- (920) -- (918) ;

\node  (921) at (10.414,1.5) {};
\node [fill={blue}] (922) at (10.39,1) {};
\node  (923) at (10.414,0.5) {};
\node  (924) at (10.789,1) {};
\draw  (922) -- (921) -- (924) -- (923) -- (922) -- (924);

\draw (0010)--(918);
\draw (919)--(921);
\draw  (920)--(923);

\end{scope}

\end{tikzpicture}
\caption{$H_{k}$ and a zero forcing set of size $3k+5$}
\label{fig:Hk}
\end{center}\vspace{-15pt}
\end{figure}

\begin{lemma} \label{lem:Hk}
For the graph $H_k$ we have $Z(H_k)=3k+5$.
\end{lemma}

\begin{proof}
The zero forcing set in Figure \ref{fig:Hk} gives the upper bound $Z(H_k)\leq 3k+5$. To establish the lower bound, we apply Lemma \ref{deletion}.  Consider the graph $(k+2)F_2\cup kT$ which consists of $k+2$ disjoint copies of $F_2$ and $k$ disjoint triangles. By Lemma \ref{lem:F2}, we have $Z((k+2)F_2\cup kT)=2(k)+3(k+2)=5k+6$. Furthermore, $(k+2)F_2\cup kT$ can be obtained from $H_k$ by deleting $2k+1$ edges: $k+2$ edges that connect the copies of $F_2$ to the triangles and $k-1$ edges that connect the triangles to one another. Therefore, Lemma~\ref{deletion} implies that $-2k-1\leq Z(H_k)-Z((k+2)F_2\cup kT)\leq 2k+1$. Using the value of $Z((k+2)F_2\cup kT)$, we have $3k+5\leq Z(H_k)\leq 7k+7$.  In particular, this gives the desired lower bound.
\end{proof}

Since $Z(H_k)=3k+5> \frac{3}{10}(10k+14)=3k+\frac{21}{5}$. This implies that the family of graphs $H_k$ act as a counter-example to inequality in question \ref{z<3/10n}.


\subsection{An upper bound for the zero forcing number of claw-free cubic graphs with Hamiltonian contraction multigraphs}

  Davila and Henning along with He et al. both consider the contraction multigraphs of claw-free cubic graphs containing only triangle units in \cite{DH} and \cite{He} respectively.  To obtain a contraction multigraph of a given claw-free cubic graph $G$ we first obtain $G'$ by removing all of the diamond units in each diamond chain and replacing them with edges.  Additionally we remove all of the diamond units in each diamond bracelet and adding an edge connecting the two link vertices of the triangle in the bracelet.  Note that $G'$ is a claw-free and diamond-free cubic graph. Then the contraction multigraph, denoted $M_{G'}$, is obtained from $G'$ by contracting every triangle unit into a single vertex.  We obtain an upper bound on the zero forcing number of claw-free cubic graphs with Hamiltonian contraction multigraphs.  We also note that if $G$ is Hamiltonian itself its contraction multigraph must also be Hamiltonian, so this bound also applies to Hamiltonian claw-free cubic graphs.\\

\begin{lemma}\label{no_cut_edge} 
If $G$ is a claw-free cubic graph such that $M_{G'}$ is Hamiltonian then $Z(G)\leq \frac{T}{2}+D+2$ where $T$ is the number of triangle units and $D$ is the number of diamond units in $G$.  
\end{lemma}
\allowdisplaybreaks
\proof Let $G$ be a Hamiltonian claw-free cubic graph. If $G$ is a diamond necklace (that is, $T=0$), then by Lemma \ref{Ncubic} $Z(G)=\frac{1}{4}n+2=D+2$ so the desired result follows immediately. Hence, we may assume for the remainder of the proof that $G$ contains at least two triangle units.

Now consider the contraction multigraph $M_{G'}$.  Since each vertex of $M_{G'}$ corresponds to a triangle in the original graph $G$, we have $|V(M_{G'})|=T$.  In addition, note that this graph is also cubic.

Since $M_{G'}$ is Hamiltonian, it contains a cycle that spans all of its vertices. Let the Hamiltonian cycle be labeled as $v_1v_2\ldots v_{T-1}v_Tv_1$, where $v_i$ is adjacent to $v_{i+1}$ modulo $T$.

Since each vertex of $M_{G'}$ has a degree of three, each $v_i$ must be adjacent to some $v_i'$ in addition to the two neighbors in the Hamiltonian cycle, so $v_i'\neq v_{i\pm1}$.
Moreover, $v_i\neq v_i'$, as otherwise either a loop would exist in the original graph $G$, or $v_i$ would correspond to the triangle unit of a diamond bracelet in $G$. The first case is impossible because $G$ does not have loops, the second case is impossible because $v_i$ is already adjacent to two distinct triangles, whereas a triangle unit of a diamond bracelet must be adjacent to two distinct diamonds, which would give $v_i$ degree four.

Therefore, the set $P=\{\{v_i, v_i'\} : 1\leq i<i'\leq T\}\subseteq E(M_{G'})$ induces a matching in $M_G'$ and $|P|=\frac{T}{2}$.

Now suppose $\{v_i, v_i'\}\in P$ and let $T_i=\{a_i, b_i, c_i\}$ and $T_i'=\{a_i', b_i', c_i'\}$ be the triangles in the graph $G$ represented by $v_i$ and $v_i'$ in $M_{G'}$. Then $T_i$ and $T_i'$ can be classified as \textit{Type A} and \textit{Type B} respectively corresponding to \textit{Case A} and \textit{Case B} of the triangle rule.  The motivation behind this is that if we initialize the forcing process at $T_1$ and follow the forcing chain along the Hamiltonian cycle, then $T_i$ will be completely filled before $T_i'$ is encountered. This allows for \textit{Case B} of the triangle rule to be applied.  

Let $\{c_i, c_i'\}\in E(G)$ be the edge corresponding to $\{v_i, v_i'\}$. To construct a zero forcing set for $G$, first color all three vertices of $T_1$ as well as the vertex $c_i$ of each $T_i$ that was classified as \textit{Type A} in the previous step. Next, for each diamond unit $D_j=\{a_j, b_j, c_j, d_j\}$ for $1\leq j\leq D$ where $\{a_j, b_j\}$ is the missing edge, color $c_j$. As established earlier, a diamond bracelet cannot exist in $G$, and therefore, every diamond unit lies on a diamond chain. Consider the forcing process along a chain from $T_1$ to $T_k$. Whenever a diamond unit $D_j$ is encountered, the entire unit will be forced through $c_j$ using the Diamond-Rule. Similarly, whenever a triangle unit $T_i$ is encountered, the entire unit will be forced through $c_i$ using the Triangle-Rule. The vertex $c_j$ will either have been colored initially (for \textit{Type A} triangles) or have been forced by a preceding triangle unit during earlier steps of the process for \textit{Type B} triangles).

Overall, this construction of a zero forcing set gives the upper bound $Z(G)\leq \frac{T}{2}+D+2$. \qed

\bigskip
{\bf Acknowledgment.}
Shahla Nasserasr was supported in part by NSF grant DMS-2532732.

\bibliographystyle{amsplain}

\begin{thebibliography}{10}

\bibitem{DH}
R.~Davila and M.~A.~Henning,
``Zero Forcing in Claw-Free Cubic Graphs,''
\emph{Bull. Malays. Math. Sci. Soc.}
\textbf{43} (2020), no.~1, 673--688.
doi:10.1007/s40840-018-00705-5.

\bibitem{HL}
M.~A.~Henning and C.~L\"owenstein,
``Locating-total domination in claw-free cubic graphs,''
\emph{Discrete Math.}
\textbf{312} (2012), no.~21, 3107--3116.

\bibitem{TS}
T.~Uno and H.~Satoh,
``An Efficient Algorithm for Enumerating Chordless Cycles and Chordless Paths,''
in \emph{Discovery Science},
Lecture Notes in Computer Science, vol.~8777,
Springer, 2014, pp.~313--324.
doi:10.1007/978-3-319-11812-3\_27.

\bibitem{He}
M.~He, H.~Li, L.~Song, and S.~Ji,
``The Zero Forcing Number of Claw-Free Cubic Graphs,''
\emph{Discrete Appl. Math.}
\textbf{359} (2024), 321--330.
doi:10.1016/j.dam.2024.08.011.

\bibitem{MR}
L.~DeLoss, J.~Grout, L.~Hogben, T.~McKay, J.~Smith, and G.~Tims,
``Table of Minimum Ranks of Graphs of Order at Most $7$ and Selected Optimal Matrices,''
\emph{arXiv preprint} arXiv:0812.0870.

\bibitem{Edge}
C.~Edholm, L.~Hogben, M.~Huynh, J.~LaGrange, and D.~Row,
``Vertex and Edge Spread of Zero Forcing Number, Maximum Nullity and Minimum Rank of a Graph,''
\emph{Linear Algebra Appl.}
\textbf{436} (2012), 4352--4372.
doi:10.1016/j.laa.2011.10.015.

\bibitem{MR2645093}
F.~Barioli, W.~Barrett, S.~M.~Fallat, H.~T.~Hall,
L.~Hogben, B.~Shader, P.~van den Driessche, and H.~van der Holst,
``Zero forcing parameters and minimum rank problems,''
\emph{Linear Algebra Appl.}
\textbf{433} (2010), no.~2, 401--411.
doi:10.1016/j.laa.2010.03.008.

\bibitem{West}
D.~B.~West,
\emph{Introduction to Graph Theory},
2nd ed.,
Prentice Hall, Upper Saddle River, NJ, 2001.

\end{thebibliography}

\end{document}